\documentclass{amsart}
\usepackage{graphicx}
\newtheorem{thm}{Theorem}[section]
\newtheorem{cor}[thm]{Corollary}
\newtheorem{lem}[thm]{Lemma}
\newtheorem{prop}[thm]{Proposition}
\theoremstyle{definition}

\theoremstyle{remark}

\numberwithin{equation}{section}

\begin{document}

\title[]{An extensions of Kannappan's and Van Vleck's functional equations on semigroups}%
\author{Elqorachi Elhoucien and Redouani Ahmed}%
\address{University Ibn Zohr, Faculty of Sciences, Department of Mathematics, Agadir, Morocco}
\email{elqorachi@hotmail.com; redouani-ahmed@yahoo.fr}


\begin{abstract}
This paper treats two functional equations, the Kannppan-Van Vleck
functional equation $$\mu(y)f(x\tau(y)z_0)\pm f(xyz_0) =2f(x)f(y),
\;x,y\in S$$ and the following variant of it
 $$\mu(y)f(\tau(y)xz_0)\pm f(xyz_0) = 2f(x)f(y), \;x,y\in
S,$$  in the setting of semigroups $S$ that need not be abelian or
unital,   $\tau$ is an involutive morphism of $S$,  $\mu$ :
$S\longrightarrow \mathbb{C}$ is a multiplicative function such that
$\mu(x\tau(x))=1$ for all $x\in S$ and $z_0$ is a fixed element in
the center of $S$. \\We find the complex-valued solutions of these
equations in terms of multiplicative functions and  solutions of
d'Alembert's functional equation.
\end{abstract}
\maketitle
\section{Introduction} Van Vleck \cite{V1,V2} studied the continuous solutions
$f$ : $\mathbb{R} \longrightarrow \mathbb{R}$, $f\neq 0$, of the
following functional equation
\begin{equation}\label{eq1}
f(x - y + z_0)-f(x + y + z_0) = 2f(x)f(y),\;x,y \in
\mathbb{R},\end{equation} where $z_0>0$ is fixed. He showed that any
continuous solution of (\ref{eq1}) with minimal period $4z_0$ is  $f(x)=\cos(\frac{\pi}{2z_0}(x-z_0))$, $x\in \mathbb{R}$.\\
Stetk\ae r \cite[Exercise 9.18]{07} found the complex-valued
solutions of equation
\begin{equation}\label{eq2} f(xy^{-1}z_0)-f(xyz_0) =
2f(x)f(y),\;x,y \in G,\end{equation}  on groups that need not  be
abelian and $z_0$ is a fixed element in the center of $G$. \\
Perkins and Sahoo \cite{P} replaced the group inversion by an
involution $\tau$: $G\longrightarrow G$ and  obtained the abelian,
complex-valued solutions of the equation
\begin{equation}\label{eq3} f(x\tau(y)z_0)-f(xyz_0) = 2f(x)f(y),\;x,y
\in G,\end{equation} by means of d'Alembert's functional equation
\begin{equation}\label{eq4} g(xy)+g(x\tau(y)) = 2g(x)g(y),\;x,y
\in G.\end{equation}\\ Stetk\ae r \cite{St3} extended the results of
\cite{P} about equation (\ref{eq3}) to semigroups and derived an
explicit formula for the solutions in terms of multiplicative maps.
In particular, Stetk\ae r proved that all solutions of (\ref{eq3})
are abelian. So,
 the
restriction to abelian solutions in \cite{P} is not
needed.\\D'Alembert's classic functional equation
\begin{equation}\label{eq5} g(x+y)+g(x-y) = 2g(x)g(y),\;x,y
\in \mathbb{R}\end{equation} has solutions  $g$:
$\mathbb{R}\longrightarrow \mathbb{C}$ that are periodic, for
instance $g(x)=\cos(x)$, and solutions that are not, for instance
$g(x)=\cosh(x)$.\\ Kannappan \cite{K} proved that any solution of
the extension of (\ref{eq5})
\begin{equation}\label{eq6}
f(x - y + z_0)+f(x + y + z_0) = 2f(x)f(y),\;x,y \in
\mathbb{R},\end{equation} where $z_0\neq0$ is a real constant has
the form $f(x)=g(x-z_0)$, where $g$: $\mathbb{R}\longrightarrow
\mathbb{C}$ is a periodic solution of (\ref{eq5}) with period
$2z_0$.\\
Perkins and Sahoo \cite{P} considered the following version of
Kannappan's functional equation
\begin{equation}\label{eq7}
f(xyz_0)+f(xy^{-1}z_0) = 2f(x)f(y),\;x,y \in G\end{equation} on
groups and they found the form of any abelian solution $f$ of
(\ref{eq7}).\\
 Stetk\ae r \cite{stkan} took $z_0$ in the center and
expressed the complex-valued solutions of Kannappan's functional
equation \begin{equation}\label{eq88} f(xyz_0)+f(x\tau(y)z_0) =
2f(x)f(y),\;x,y \in S\end{equation} on semigroups with involution
$\tau$ in terms of solutions of d'Alembert's functional equation
(\ref{eq4}).\\In the very special case of $z_0$ being the neutral
element of a monoid $S$ equation (\ref{eq88}) becomes (\ref{eq4})
which has been solved by Davison \cite{davison}.\\Here we shall
consider the following functional equations
\begin{equation}\label{eq8}
f(xyz_0)+\mu(y)f(x\tau(y)z_0) =2f(x)f(y), \;x,y\in S,
\end{equation}
\begin{equation}\label{eq9}
f(xyz_0)+\mu(y)f(\tau(y)xz_0) =2f(x)f(y), \;x,y\in S,
\end{equation}
\begin{equation}\label{eq10}
\mu(y)f(x\tau(y)z_0)-f(xyz_0) =2f(x)f(y), \;x,y\in S
\end{equation}
and \begin{equation}\label{eq11} \mu(y)f(\tau(y)xz_0)-f(xyz_0)
=2f(x)f(y), \;x,y\in S,\end{equation} where  $S$ is a semigroup,
$\tau$ is an involutive morphism of $S$. That is, $\tau$  is an
involutive automorphism: $\tau(xy)=\tau(x)\tau(y)$ and
$\tau(\tau(x))=x$ for all $x,y\in S$ or  $\tau$ is an  involutive
anti-automorphism: $\tau(xy)=\tau(y)\tau(x)$ and $\tau(\tau(x))=x$
for all $x,y\in S.$ The map $\mu$ : $S\longrightarrow \mathbb{C}$ is
a multiplicative function such that $\mu(x\tau(x))=1$ for all $x\in
S$ and $z_0$ is a fixed element in the center of $S$. By algebraic
methods: \\(1) We find all solutions of (\ref{eq10}) and
(\ref{eq11}). Only multiplicative functions occur in the solution
formulas.\\(2) We find the solutions of (\ref{eq9}) for the
particular case of $\tau$ being an involutive automorphism and \\3)
We express the solutions of (\ref{eq8}) and (\ref{eq9}) in terms of
solutions of d'Alembert's $\mu$-functional equations
\begin{equation}\label{eq12} g(xy)+\mu(y)g(x\tau(y))
=2g(x)g(y), \;x,y\in S.\end{equation} Of course we are not the first
 to consider trigonometric functional equations having a
multiplicative function $\mu$ in front of   terms like $f(x\tau(y))$
or $f(\tau(y)x)$. The $\mu$-d'Alembert's functional equation
(\ref{eq12}) which is an extension of d'Alembert's functional
equation (\ref{eq4})   has been treated systematically by Stetk\ae r
\cite{St1, 07} on groups with involution. The non zero solutions of
(\ref{eq12}) on groups with involution are the normalized traces of
certain representation of $S$ on $\mathbb{C}^{2}$.\\
Stetk\ae r \cite{St100} obtained the complex-valued solution of the
following variant of \\d'Alembert's functional equation
\begin{equation}\label{om1}
f(xy) + f(\tau(y)x)) = 2f(x)f(y), \;x, y \in S ,\end{equation} where
$\tau$ is an involutive automorphism of $S$.
\\ Elqorachi and Redouani \cite{elqorachi} proved that the solutions
of  the variant of d'Alembert's functional equation
\begin{equation}\label{om2} f(xy) + \mu(y)f(\tau(y)x)) = 2f(x)f(y), \;x, y
\in S\end{equation}  are of the form $f(x) =
\frac{\chi(x)+\mu(x)\chi(\tau(x))}{2}$, $x\in S$, where $\tau$ is an
involutive automorphism of $S$ and  $\chi$: $S\longrightarrow S$ is
a multiplicative function.
\\Bouikhalene and Elqorachi \cite{preprint} obtained the solutions
of (\ref{eq10}) for  involutive anti-automorphism $\tau$. In the
same paper they also found  the solutions of (\ref{eq10}) for
 involutive automorphism $\tau$, but on monoids only.\\Throughout
 this paper $S$ denotes a semigroup with an involutive morphism $\tau$: $S\longrightarrow S$, $\mu$: $S\longrightarrow \mathbb{C}$  denotes a
 multiplicative function such that $\mu(x\tau(x))=1$ for all $x\in S$  and $z_0$ a fixed element in the center of $S$.
 \\In all proofs of the results of this paper we use without explicit
mentioning the assumption that $z_0$ is contained in the center of
$S$   and its consequence $\tau(z_0)$ is contained in the center of
$S$.
 \section{Solutions of equation (\ref{eq8}) on semigroups}
 In this section we express the solutions of (\ref{eq8}) in terms of
 solutions of d'Alembert's functional equation (\ref{eq12}). The
 following lemma will be used later.
\begin{lem}
If $f$: $S\longrightarrow \mathbb{C}$ is a solution of (\ref{eq8}),
then  for all $x\in S$
\begin{equation}\label{eqo77}
    f(x)=\mu(x)f(\tau(x)),
\end{equation}
\begin{equation}\label{eqo999}
    f(x\tau(z_0)z_0) =\mu(\tau(z_0))f(z_0)f(x),
\end{equation}
\begin{equation}\label{eqo1000}
    f(xz_{0}^{2})=f(x)f(z_0),
\end{equation}
\begin{equation}\label{eqo88}
    f(z_0)\neq 0\Longleftrightarrow f\neq 0.
\end{equation}
  \end{lem}
  \begin{proof}Equation (2.1):  By replacing $y$ by $\tau(y)$ in (\ref{eq8}) and multiplying the
   result obtained by $\mu(y)$ and using $\mu(y\tau(y))=1$ we get by
  computation that $$\mu(y)2f(x)f(\tau(y))=\mu(y)f(x\tau(y)z_0)+\mu(y\tau(y))f(xyz_0)$$
  $$=\mu(y)f(x\tau(y)z_0)+f(xyz_0)=2f(x)f(y),$$ which implies
  (2.1).\\ Equation (2.2): Replacing $x$ by  $\tau(z_0)$  in (\ref{eq8}) and  using
  (2.1) two times we get by a computations that
  $$f(\tau(z_0)yz_0)+\mu(y)f(\tau(z_0)\tau(y)z_0)=2f(\tau(z_0))f(y)=2\mu(\tau(z_0))f(z_0)f(y)$$and
  $$f(\tau(z_0)yz_0)+\mu(y)f(\tau(z_0)\tau(y)z_0)=2f(\tau(z_0))f(y)$$$$=f(\tau(z_0)yz_0)+\mu(y)\mu(\tau(z_0)\tau(y)z_0)f(\tau(z_0)yz_0)=2f(\tau(z_0)yz_0).$$ This proves (2.2).
  \\  Equation (2.3): Putting $y = z_0$ in (\ref{eq8}) and using
(2.2) we obtain (2.3).\\ Equation (2.4): Assume that $f(z_0) = 0$.
By replacing $x$ by $xz_0$ and $y$ by $yz_0$ in (\ref{eq8}) and
using (2.2) and (2.3) we get by a computation that
$$2f(xz_0)f(yz_0)=f(xz_0yz_{0}^{2})+\mu(yz_0)f(xz_0\tau(y)\tau(z_0)z_0)$$
$$=f(z_0)f(xyz_0)+\mu(y)f(z_0)f(x\tau(y)z_0)=0\;\text{for
all }\;x,y\in S.$$ Which implies that $f(xz_0)=0$ for all $x\in S$.
So, from equation (\ref{eq8}) we get $2f(x)f(y)=0$ for all $x,y\in
S,$ and then $f(x)=0$ for all $x\in S$. Conversely, it's clear that
$f(x)=0$ for all $x\in S$ implies that $f(z_0)=0$.\end{proof}
For the rest of this section we use the following notations \cite{stkan}.\\
- $\mathcal{A}$ consists of the solutions of $g:$ $S\longrightarrow
\mathbb{C}$ of d'Alembert's functional equation (\ref{eq12}) with $
g(z_0) \neq 0$ and satisfying the condition
\begin{equation}\label{pr2}
 g(xz_0)=g(z_0)g(x)\; \text{for all}\;x\in S.\end{equation}
\\- To any $g\in \mathcal{A}$ we associate the function
$Tg=g(z_0)g:$ $S\longrightarrow \mathbb{C}$.\\
- $\mathcal{K}$ consists of the non-zero solutions $f:$
$S\longrightarrow \mathbb{C}$ of  Kannappan's functional equation
(\ref{eq8}).
\\In the following main result of the present section, the complex
solutions of equation (\ref{eq8}) are expressed by means of
solutions of d'Alembert's functional equation (\ref{eq12}).
\begin{thm} (1)
$T$ is a bijection of $\mathcal{A}$ onto $\mathcal{K}$. The inverse
$T^{-1}$: $\mathcal{K}\longrightarrow \mathcal{A}$ is given by the
formula
$$(T^{-1}f)(x)=\frac{f(xz_0)}{f(z_0)}$$ for
all $f\in \mathcal{K}$ and $x\in S. $\\{(2)} Any non-zero solution
$f$: $S\longrightarrow \mathbb{C}$ of the  Kannappan's functional
equation (\ref{eq8}) is of the form $f=T(g)=g(z_0)g$, where $g\in
\mathcal{A}$. Furthermore,
$$f(x)=g(xz_0)=\mu(z_0)g(x\tau(z_0))=g(z_0)g(x)$$ for all $x\in
S.$\\(3) $f$ is central i.e. $f(xy)=f(yx)$ for all $x,y\in S$ if and
only if $g$ is central.\\(4) $f$ is abelian [\cite{07}, Definition
B.3.] if and only if $g$ is abelian.\\(5) If $S$ is equipped with a
topology then $f$ is continuous if and only if $g$ is continuous.
\end{thm}
\begin{proof}For any $g\in \mathcal{A}$ and for all $x,y\in S$ we have $$T(g)(xyz_0)+\mu(y)T(g)(x\tau(y)z_0)=g(z_0)[g(xyz_0)+\mu(y)g(x\tau(y)z_0)]$$
$$=g(z_0)^{2}[g(xy)+\mu(y)g(x\tau(y))]=2g(z_0)g(x)g(z_0)g(y)=2T(g)(x)T(g)(y).$$  On the other hand $T(g)(z_0)=g(z_0)^{2}\neq0$, so we get
  $T(\mathcal{A})\subseteq \mathcal{K}.$\\
By adapting the proof of  [\cite{stkan}, Lemma 3]  $T$ is injective.
Now, we will show that $T$ is surjective. Let $f\in \mathcal{K}$.
Then from (2.4) we have $f(z_0)\neq0$ and we can define the function
$g(x)=\frac{f(xz_0)}{f(z_0)}.$ In the following we will show that
$g\in \mathcal{A}$ and $T(g)=f$. By using the definition of $g$ and
(2.2)-(2.3)  we have
$$f(z_0)^{2}[g(xy)+\mu(y)g(x\tau(y))]=f(z_0)f(xyz_0)+\mu(y)f(z_0)f(x\tau(y)z_0)$$
$$=f(xyz_{0}^{3})+\mu(y)\mu(z_0)f(x\tau(y)z{_0}^{2}\tau(z_0))$$
$$=f(xz_0yz_0z_0)+\mu(yz_0)f(xz_0\tau(yz_0)z_0)=2f(xz_0)f(yz_0)=2f(z_0)^{2}g(x)g(y)$$
for all $x,y\in S.$ This shows that $g$ is a solution of
d'Alembert's functional equation (\ref{eq12}).
 \\By replacing $x$ by $xz_{0}^{2}$ and $y$ by $z_0$ in
(\ref{eq8})  we get
\begin{equation}\label{red1}
f(xz_{0}^{4})+\mu(z_0)f(xz_{0}^{3}\tau(z_0))=2f(xz_{0}^{2})f(z_0).\end{equation}
By replacing $x$  by $xz_0$ and $y$ by $z_{0}^{2}$ in (\ref{eq8}) we
have
\begin{equation}\label{red2}
f(xz_{0}^{4})+\mu(z_{0}^{2})f(xz_{0}^{2}\tau(z_{0}^{2}))=2f(z_{0}^{2})f(xz_0).\end{equation}
From (2.2) and (2.3) we have
$$f(xz_{0}^{3}\tau(z_0)) =\mu(\tau(z_0))f(x)(f(z_0))^{2}$$ and
$$f(xz_{0}^{2}\tau(z_{0}^{2}) ) =(\mu(\tau(z_0)))^{2}f(x)(f(z_0))^{2}.$$ In view of (\ref{red1})
and (\ref{red2}) we deduce that $f(z_{0}^{2}) f(xz_0) =
f(xz_{0}^{2})f(z_0).$ So, by using the definition of $g$ we obtain $
g(xz_0) =g(x)g(z_0)$ for all $x\in S.$ In particular,
$g(z_{0}^{2})=g(z_0)^{2}=\frac{f(z_{0}^{2}z_0)}{f(z_0)}=\frac{f(z_{0})f(z_0)}{f(z_0)}=f(z_0)\neq0.$
Furthermore,
$T(g)(x)=g(z_0)g(x)=g(xz_0)=\frac{f(xz_{0}^{2})}{f(z_0)}=\frac{f(x)f(z_{0})}{f(z_0)}=f(x)$.\\
The  statements (2), (3), (4) and (5) are obvious. This completes
the proof.
\end{proof}
Now, we extend Stetk\ae r's result \cite{stkan} from
anti-automorphisms to the more general case of morphism as follows.
\begin{cor}Let $z_0$ be a fixed element in the center of a semigroup $S$ and let  $\tau$ be an involutive morphism of $S$. Then, any non-zero solution
$f$: $S\longrightarrow \mathbb{C}$ of  the functional equation (1.8)
is of the form $f=g(z_0)g$, where $g$ is a
 solution of d'Alembert's functional equation (1.4)
with $g(z_0)\neq 0$ and satisfying the condition
$g(xz_0)=g(z_0)g(x)$ for all $x\in S$.\end{cor}We will in the
following propositions determine all abelian (resp. central)
solutions $f$ of  Kannappan's functional equation (1.9).
\begin{prop} Let $z_0$ be a fixed element in the center of a semigroup $S$. Let  $\tau$: $S\longrightarrow S$ be an involutive anti-automorphism of $S$
 and let $\mu$: $S\longrightarrow \mathbb{C}$  be a multiplicative function such that $\mu(x\tau(x))=1$ for all $x\in S.$  The non-zero abelian solutions of
  Kannappan's functional equation (1.9) are the functions of the
form
$$f(x)=\frac{\chi(x)+\mu(x)\chi(\tau(x))}{2}\chi(z_0),\; x\in
S,$$ where $\chi:$ $S\longrightarrow \mathbb{C}$ is a multiplicative
function such that $\chi(z_0)\neq 0$ and
$\mu(z_0)\chi(\tau(z_0))=\chi(z_0)$.
\end{prop}\begin{proof}Verifying that the function $f$ defined in Proposition 2.4
is an abelian solution of (1.9) consists of simple computations that we omit.\\
Let $f$: $S\longrightarrow \mathbb{C}$ be a non-zero solution of
(1.9). From Theorem 2.2(2) and (4) the function $f$ has the form
$f=g(z_0)g$ where $g\in \mathcal{A}$ and $g$ is abelian. From
[\cite{belfkih}, Theorem 2.1], [12, Proposition 9.31] there exists a
non-zero multiplicative function $\chi$: $S\longrightarrow
\mathbb{C}$ such that $g=\frac{\chi+\mu\chi\circ\tau}{2}$. Since
$g\in \mathcal{A}$, it satisfies (2.5).  If we replace $x$ by $z_0$
in (2.5) we get $g(z_{0}^{2})=g(z_0)^{2}$, which via computation
gives that  $\chi(z_0)=\mu(z_0)\chi(\tau(z_0))$. This implies that
$f$ has the desired form. This completes the proof.
\end{proof}By using [\cite{elqorachi}, Lemma 3.2] and the proof of the
preceding proposition we get
\begin{prop} Let $z_0$ be a fixed element in the center of a semigroup $S$. Let  $\tau$: $S\longrightarrow S$ be an involutive automorphism of $S$
 and let $\mu$: $S\longrightarrow \mathbb{C}$  be a multiplicative function such that $\mu(x\tau(x))=1$ for all $x\in S.$  The non-zero central solutions of
the  Kannappan's functional equation (1.9) are the functions of the
form $$f(x)=\frac{\chi(x)+\mu(x)\chi(\tau(x))}{2}\chi(z_0),\; x\in
S,$$ where $\chi:$ $S\longrightarrow \mathbb{C}$ is a multiplicative
function such that $\chi(z_0)\neq 0$ and
$\mu(z_0)\chi(\tau(z_0))=\chi(z_0)$.
\end{prop}
\section{Solutions of equation (\ref{eq9}) on semigroups}In this
section we determine the complex-valued solutions of (\ref{eq9}) for
any involutive morphism  $\tau$: $S\longrightarrow S$. By help of
Theorem 2.2 we express them  in terms of solutions of d'Alembert's
functional equation (1.13). We first prove the following two useful
lemmas.
\begin{lem} If $f$:
$S\longrightarrow \mathbb{C}$ is a solution of (\ref{eq9}), then for
all $x\in S$
\begin{equation}\label{sal1}
    f(x)=\mu(x)f(\tau(x)),
\end{equation}
\begin{equation}\label{sal2}
    f(x\tau(z_0)z_0) =\mu(\tau(z_0))f(z_0)f(x),
\end{equation}
\begin{equation}\label{sal3}
    f(xz_{0}^{2})=f(x)f(z_0),
\end{equation}
\begin{equation}\label{sal4}
    f(z_0)\neq 0\Longleftrightarrow f\neq 0.
\end{equation}
  \end{lem}
  \begin{proof} Equation (3.1): Interchanging $x$ and $y$ in (\ref{eq9}) and multiplying the two members of the equation by $\mu(\tau(y))$   we
  get \begin{equation}\label{sal5}
\mu(x)\mu(\tau(y))f(\tau(x)yz_0)+\mu(\tau(y))f(yxz_0)
=2f(x)\mu(\tau(y))f(y), \;x,y\in S.
\end{equation}Replacing  $y$ by $\tau(y)$ in (\ref{eq9}) we
  obtain \begin{equation}\label{sal6}
\mu(\tau(y))f(yxz_0)+f(x\tau(y)z_0) =2f(x)f(\tau(y)), \;x,y\in S.
\end{equation}By subtracting (\ref{sal6}) from (\ref{sal5}) we get
\begin{equation}\label{sal7}
\mu(x\tau(y))f(\tau(x)yz_0)-f(x\tau(y)z_0)
=2f(x)[\mu(\tau(y))f(y)-f(\tau(y)], \;x,y\in S.\end{equation}By
replacing $x$ by $\tau(x)$ in (\ref{sal7}) we have
\begin{equation}\label{sal8}
\mu(\tau(x)\tau(y))f(xyz_0)-f(\tau(x)\tau(y)z_0)
=2f(\tau(x))[\mu(\tau(y))f(y)-f(\tau(y)], \;x,y\in S.\end{equation}
Replacing $y$ by $\tau(y)$ in (\ref{sal7}) and multiplying the two
members of the equation by $\mu(\tau(y)\tau(x))$ we obtain
\begin{equation}\label{sal9}
f(\tau(x)\tau(y)z_0)-\mu(\tau(x)\tau(y))f(xyz_0)
=2f(x)\mu(\tau(x))[f(\tau(y))-\mu(\tau(y))f(y)], \;x,y\in
S.\end{equation}Now, by adding (\ref{sal8}) and (\ref{sal9}) we get
$[f(\tau(x))-\mu(\tau(x))f(x)][f(\tau(y))-\mu(\tau(y))f(y)]=0$ for
all $x,y\in S.$ This proves (3.1).\\Equation (3.2): Taking
$x=\tau(z_0)$ in (\ref{eq9}) and using (3.1) we get
$$\mu(y)f(\tau(y)\tau(z_0)z_0)+f(\tau(z_0)yz_0)=2\mu(\tau(z_0))f(z_0)f(y)$$
$$=f(\tau(z_0)yz_0)+\mu(y)\mu(\tau(y)\tau(z_0)z_0)f(\tau(z_0)yz_0)=2f(\tau(z_0)yz_0).$$
Which implies (3.2).\\
Equation (3.3): By replacing $y$ by $z_0$ in (\ref{eq9}) and using
(3.2) we obtain
$$\mu(z_0)f(\tau(z_0)xz_0)+f(x{z_0}^{2})=2f(z_0)f(x)$$
$$=\mu(z_0)\mu(\tau(z_0))f(z_0)f(x)+f(x{z_0}^{2}).$$So, we deduce
(3.3).\\Equation (3.4): The proof is similar to  the proof of
(2.4).\end{proof}
\begin{lem}Let $\mathcal{M}$ consist of the solutions $g$: $S\longrightarrow \mathbb{C}$ of the variant d'Alembert's functional equation (1.15) with $g(z_0)\neq0$
and satisfying the condition (2.5). Let $\mathcal{N}$ consist of the
non-zero solutions $f$: $S\longrightarrow \mathbb{C}$ of the variant
Kannappan's functional equation (1.10); Then \\(1) The map $J$:
$\mathcal{M}\longrightarrow \mathcal{N}$  defined by $Jh:=h(z_0)h$:
$S\longrightarrow \mathbb{C}$ is a bijection. The inverse $J^{-1}$;
$\mathcal{N}\longrightarrow \mathcal{M}$ is given by the formula
$(J^{-1}f)(x)=\frac{f(xz_0)}{f(z_0)}=g(x)$ for all $x\in S$ and for
all $f\in \mathcal{N}$.  Furthermore, \\(2) If $\tau$:
$S\longrightarrow S$ is an involutive automorphism the function $g$
has the form $g=\frac{\chi+\mu\chi\circ \tau}{2}$, where $\chi$ :
$S\longrightarrow \mathbb{C}$, $\chi\neq 0$, is a multiplicative
function.\\(3) If $\tau$: $S\longrightarrow S$ is an involutive
anti-automorphism the function $g$ satisfies the d'Alembert's
functional equation (1.13).
\end{lem}
\begin{proof} For all $h\in\mathcal{ M}$ we have $$Jh(xyz_0)+\mu(y)Jh(\tau(y)xz_0)=h(z_0)h(xyz_0)+\mu(y)h(z_0)h(\tau(y)xz_0)$$
$$=h(z_0)^{2}[h(xy)+\mu(y)h(\tau(y)x)]=2h(z_0)h(x)h(z_0)h(y)=2Jh(x)Jh(y).$$
Furthermore, $Jh(z_0)=h(z_0)^{2}\neq0$. So, $Jh\in \mathcal{N}.$ By
adapting the proof of [\cite{stkan}, Lemma 3] $J$ is injective. Now,
let $f\in \mathcal{N}$ and let $g(x):=\frac{f(xz_0)}{f(z_0)}$ for
$x\in S.$ By using the definition of $g$, equations (1.10), (3.2)
and (3.3) we get
$$f(z_0)^{2}[g(xy)+\mu(y)g(\tau(y)x)-2g(x)g(y)]$$
$$=f(z_0)f(xyz_0)+\mu(y)f(z_0)f(\tau(y)xz_0)-2f(xz_0)f(yz_0)$$$$=f(xyz_0z_{0}^{2})+\mu(y)\mu(z_0)f(\tau(y)xz_0\tau(z_0)z_0)-2f(xz_0)f(yz_0)$$
$$f(xz_0yz_0z_0)+\mu(yz_0)f(\tau(yz_0)xz_0z_0)-2f(xz_0)f(yz_0)=0.$$ Since
$f(z_0)\neq0$ then $g$ satisfies (1.15). By using similar computations as in the proof of Theorem 2.2 we get that $g(xz_0)=g(z_0)g(x)$ for all $x\in S.$\\
(2) If $\tau$: $S\longrightarrow S$ is an involutive automorphism
then from [\cite{elqorachi}, Lemma 3.2]  $g$ from has the form
$g=\frac{\chi+\mu\chi\circ \tau}{2}$, where $\chi$ :
$S\longrightarrow \mathbb{C}$, $\chi\neq 0$, is a multiplicative
function.\\(3) If $\tau$: $S\longrightarrow S$ is an involutive
anti-automorphism then by adapting the proof of  [\cite{elqorachi},
Theorem 2.1(1)(i)] for $\delta=0$ we get   $g$ satisfies the
d'Alembert's functional equation (1.13).
   \end{proof}
   \begin{thm}
   (1) Let $\tau$: $S\longrightarrow S$ be an involutive automorphism. The non-zero solutions $f$ : $S\longrightarrow \mathbb{C}$ of the functional equation (\ref{eq9})
are the functions of the form
\begin{equation}\label{eq300}
    f=\frac{\chi+ \mu\chi\circ\tau}{2}\chi(z_0),
\end{equation}where $\chi$ : $S\longrightarrow \mathbb{C}$ is a
multiplicative function such that $\chi(z_0)\neq 0$ and
$\mu(z_0)\chi(\tau(z_0))=\chi(z_0)$. \\
(2)  Let $\tau$: $S\longrightarrow S$ be an involutive
anti-automorphism. The non-zero solutions $f$ : $S\longrightarrow
\mathbb{C}$ of the functional equation (\ref{eq9}) are the functions
of the form $f=g(z_0)g$,  where $g$ is a
 solution of d'Alembert's functional equation (1.13)
with $g(z_0)\neq 0$ and satisfying the condition
$g(xz_0)=g(z_0)g(x)$ for all $x\in S$.
\end{thm}
    \begin{proof}Let $f$: $S\longrightarrow S$ be a non-zero solution of equation (\ref{eq9}). From Theorem 2.2 (2) $f=g(z_0)g(x)=g(xz_0)$, where $g$ is a solution of d'Alembert's
     functional equation (1.4). We will discuss two possibilities.
    \\ (1) $\tau$ is an involutive automorphism of $S$. From Lemma 3.2 there exists $\chi:$ $S\longrightarrow \mathbb{C}$ a multiplicative function such that
    $g=\frac{\chi+ \mu\chi\circ\tau}{2}$. So,
    \begin{equation}\label{rkia}
    f=g(z_0)=\frac{\chi+ \mu\chi\circ\tau}{2}g(z_0)=\frac{\chi(z_0)+ \mu(z_0)\chi\circ\tau(z_0)}{2}\frac{\chi+ \mu\chi\circ\tau}{2}.\end{equation}
     By using $g(z_{0}^{2})=g(z_0)^{2}$ we get after simple computation
     that $\chi(z_0)=\mu(z_0)\chi(\tau(z_0))$. This proves (1).\\
      (2) $\tau$ is an involutive anti-automorphism of $S$. Combining Theorem 2.2 and Lemma 3.2 (2) we find (2). This completes the proof.\end{proof}
     \section{Solutions of
 equation (\ref{eq10})} The solutions of the functional equation
(\ref{eq10}) with $\tau$ an involutive anti-automorphism are
explicitly obtained  by Bouikhalene and Elqorachi \cite{preprint} on
semigroups not necessarily abelian in terms of multiplicative
functions. In this section we obtain a similar formula  for the
solutions  of the functional equation  (\ref{eq10}) when  $\tau$ was
an involutive automorphism. The following lemma is obtained in
\cite{preprint} for the case where $\tau$ is an involutive
anti-automorphism. It still holds for the case where $\tau$ is an
involutive automorphism.
\begin{lem} Let $f \neq 0$ be a solution of (\ref{eq10}). Then for
all $x\in S$ we have
\begin{equation}\label{eqq7}
    f(x)=-\mu(x)f(\tau(x)),
\end{equation}
\begin{equation}\label{eqq8}
    f(z_0)\neq 0,
\end{equation}
\begin{equation}\label{eq9'}
    f(z_{0}^{2})=0,
\end{equation}
\begin{equation}\label{eqq9}
    f(x\tau(z_0)z_0)=\mu(\tau(z_0))f(x)f(z_0),
\end{equation}
\begin{equation}\label{eqq10}
    f(xz_{0}^{2})=-f(z_0)f(x),
\end{equation}
\begin{equation}\label{eqq11}
    \mu(x)f(\tau(x)z_0)=f(xz_0).
\end{equation} The function $g(x)=\frac{f(xz_0)}{f(z_0)}$ is a
non-zero solution of d'Alembert's functional equation (1.13).
 \end{lem}
Now, we are ready to prove the  main result of this section.\\
In \cite{preprint} we used [\cite{07}, Proposition 8.14] to prove
that the function $g$ defined in Lemma 4.1 is an abelian solution of
(1.13), where $\tau$ is an involutive anti-automorphism of $S$. This
reasoning no longer works for the present situation. We will use
another approach.
\begin{thm} The non-zero solutions $f$ : $S\longrightarrow \mathbb{C}$ of
the functional equation (\ref{eq10}), where $\tau$ is an involutive
morphism of $S$ are the functions of the form
\begin{equation}\label{eq300}
    f=\chi(z_0)\frac{\mu\chi\circ\tau-\chi}{2},
\end{equation}where $\chi$ : $S\longrightarrow \mathbb{C}$ is a
multiplicative function such that $\chi(z_0)\neq 0$ and
$\mu(z_0)\chi(\tau(z_0))=-\chi(z_0)$.\\If $S$ is a topological
semigroup and that $\tau$ : $S\longrightarrow S$, $\mu$ :
$S\longrightarrow \mathbb{C}$ are continuous, then the non-zero
solution $f$ of equation (\ref{eq10}) is continuous, if and only if
$\chi$ is continuous.
\end{thm}
\begin{proof} Let $f$ be a non-zero solution of
(\ref{eq10}). Replacing $x$ by $xz_0$ in (\ref{eq10}) and using
(\ref{eqq10}) we get
\begin{equation}\label{haj2}
-\mu(y)f(x\tau(y))+f(xy)=2f(y)g(x),\;x,y\in S,\end{equation} where
$g$ is the function defined in Lemma 4.1.\\
If we replace $y$ by $yz_0$ in (\ref{eq10}) and use (\ref{eqq10})
and (\ref{eqq9})  we get
\begin{equation}\label{haj4}
\mu(yz_0)\mu(\tau(z_0))f(x\tau(y))+f(xy)=2f(x)g(y)=\mu(y)f(x\tau(y))+f(xy),\;x,y\in
S.\end{equation} By adding (4.8) and (4.9) we get that the pair
$f,g$ satisfies the sine addition law
$$f(xy)=f(x)g(y)+f(y)g(x)\; \text{for all }\; x,y\in S.$$ Now, in
view of [\cite{ebanks}, Lemma 3.4], [\cite{07}, Theorem 4.1]    $g$
is abelian. Since $g$ is a non-zero solution of d'Alembert's
functional equation (1.13) so, from [\cite{belfkih}, Theorem 2.1]
there exists a non-zero multiplicative function $\chi$:
$S\longrightarrow \mathbb{C}$ such that
$g=\frac{\chi+\mu\chi\circ\tau}{2}$. The rest of the proof is
similar to the one used in  \cite{preprint}.
\end{proof}
\section{Solutions of equation (\ref{eq11}) } The solutions of
(\ref{eq11}) were obtained in \cite{preprint} on monoids for $\tau$
an involutive automorphism. In this section  we determine the
solutions of (\ref{eq11}) for the general case where $S$ is assumed
to be a semigroup and $\tau$  an involutive morphism of $S$.\\The
following useful lemmas  will be used later.
\begin{lem} Let $f$: $S\longrightarrow
\mathbb{C}$ be a  solution of equation (\ref{eq11}). Then for all
$x,y\in S$ we have
\begin{equation}\label{eqo7}
    f(x)=-\mu(x)f(\tau(x)),
\end{equation}
\begin{equation}\label{eqo8}
f\neq 0\Longleftrightarrow f(z_0)\neq0,
\end{equation}
\begin{equation}\label{eqoo8}
\mu(y)f(\tau(y)x)=-\mu(x)f(\tau(x)y),
\end{equation}
\begin{equation}\label{eqo9}
     f(x\tau(z_0)z_0)=\mu(\tau(z_0))f(z_0)f(x) ,
\end{equation}
\begin{equation}\label{eqo10}f(xz_{0}^{2})=-f(z_0)f(x),
\end{equation}
\begin{equation}\label{eqo11}
     \mu(x) f(\tau(x)z_0)= f(xz_0),
\end{equation}
\begin{equation}\label{elq10}
 f(x\tau(z_0))=\mu(x)f(\tau(x)\tau(z_0)),\end{equation}
 \begin{equation}\label{omar100}
     f(z_{0}^{2})=f(z_0\tau(z_0))=0.
\end{equation}
 \end{lem}\begin{proof} Equation (5.2): Let $f\neq 0$  be a non-zero solution of equation
 (\ref{eq11}). We will derive (5.2) by contradiction. Assume that $f(z_0)=0$. Putting  $y=z_0$ in
 equation (\ref{eq11})  we get
\begin{equation}\label{omar1}
 \mu(z_0)f(\tau(z_0)xz_{0})-f(xz_0z_{0}) =
 2f(x)f(z_0)=0\end{equation}
Replacing  $y$ by $yz_0$ in (\ref{eq11}) and  using (\ref{omar1})
and (\ref{eq11}) we get
 $$\mu(yz_0)f(\tau(y)xz_0\tau(z_{0}))-f(xyz_{0}z_0) = 2f(x)f(yz_{0})$$
 $$=\mu(y)f(\tau(y)xz_0z_{0})-f(xz_{0}yz_0) $$
 $$= 2f(y)f(xz_0).$$
 So, we deduce that $f(y) f(xz_{0})=f(x)
 f(yz_{0})$ for all $x,y\in S.$ Since $f\neq 0$, then
 there exists $\alpha\in \mathbb{C}$ such that
  $f(xz_{0})=\alpha
 f(x)$ for all $x\in S.$ Furthermore, $\alpha\neq0$, because if $\alpha=0$ we get $f(xz_0)=0$ for all $x\in S$
 and equation (\ref{eq11}) implies that $f=0$. This contradicts
 the assumption that $f\neq0$. \\Now,  by substituting  $f(xz_{0})=\alpha
 f(x)$ into (\ref{eq11}) we get
 \begin{equation}\label{elq1}
   \mu(y)f(\tau(y)x)-f(xy)=\frac{2}{\alpha}f(x)f(y)\; \text{for all }\;x,y \in S.
 \end{equation}
Switching $x$ and $y$ in (\ref{elq1}) we get
\begin{equation}\label{mohhh1}
- f(yx)+\mu(x)f(\tau(x)y)=\frac{2}{\alpha}f(x)f(y),\;  x,y \in S.
\end{equation} If we replace $y$ by $\tau(y)$ in (\ref{elq1}) and multiplying the result obtained by $\mu(y)$ we
get \begin{equation}\label{mohhh2}
 -\mu(y)f(x\tau(y))+f(yx)=\frac{2}{\alpha}f(x)\mu(y)f(\tau(y)),\;  x,y \in S.
\end{equation} By adding    (\ref{mohhh2}) and (\ref{mohhh1})  we obtain
\begin{equation}\label{mohhh3}
 -\mu(y)f(x\tau(y))+\mu(x)f(\tau(x)y)=\frac{2}{\alpha}f(x)[\mu(y)f(\tau(y))+f(y)],\;  x,y \in S.
\end{equation} By replacing $x$ by $\tau(x)$ in (\ref{mohhh3}) and multiplying the result obtained by $\mu(x)$ we
get \begin{equation}\label{mohhh4}
 f(xy)-\mu(xy)f(\tau(x)\tau(y))=\frac{2}{\alpha}\mu(x)f(\tau(x))[\mu(y)f(\tau(y))+f(y)].
\end{equation} By replacing $y$ by $\tau(y)$ in (\ref{mohhh3}) and multiplying the result obtained by $\mu(y)$ we
get \begin{equation}\label{mohhh5}
 \mu(xy)f(\tau(x)\tau(y))-f(xy)=\frac{2}{\alpha}f(x)[f(y)+\mu(y)f(\tau(y))].
\end{equation}By adding   (\ref{mohhh4}) and
(\ref{mohhh5})  we obtain
\begin{equation}\label{mohhh6}
 [f(x)+\mu(x)f(\tau(x))][\mu(y)f(\tau(y))+f(y)]=0,\;  x,y \in S.
\end{equation} So,  $\mu(x)f(\tau(x))=-f(x)$ for all $x\in S.$ Now, we will discuss
the following two  cases.\\ (1)  $\tau$ is an involutive
anti-automorphism.  By using $\mu(x)f(\tau(x))=-f(x)$ for all $x\in
S$ we get $f(\tau(y)x)=-\mu(\tau(y)x)f(\tau(x)y)$ for all $x,y\in
S$. Substituting this in equation
 (\ref{elq1}) we obtain
 \begin{equation}\label{elq2}
   f(xy)+\mu(x)f(\tau(x)y)=2\frac{-f(x)}{\alpha}f(y),\;x,y \in S.
 \end{equation}By replacing $x$ by $\tau(x)$ in (5.17) and multiplying the result obtained by $\mu(x)$ we deduce that  $f(x)=\mu(x)f(\tau(x))$
 for all $x\in S$. So, we have
 $f(x)=-\mu(x)f(\tau(x))=-f(x)$, which  implies that $f=0$. This contradicts the
assumption that $f\neq 0$.
\\(2)   $\tau$ is an involutive
automorphism. Then from $\mu(x)f(\tau(x))=-f(x)$ for all $x\in S$ we
get $f(\tau(y)x)=-\mu(\tau(y)x)f(y\tau(x))$ for all $x,y\in S$.
Substituting this in equation
 (\ref{elq1}) we obtain
 \begin{equation}\label{elq2}
   f(xy)+\mu(x)f(y\tau(x))=2\frac{-f(x)}{\alpha}f(y)\; \text{for all }\;x,y \in S.
 \end{equation}By replacing $x$ by  $\tau(x)$ in (\ref{elq2}) and multiplying the result obtained by $\mu(x)$ and using $\mu(x)f(\tau(x))=-f(x)$
 we get $$h(yx)+\mu(x)h(\tau(x)y)=2h(x)h(y)\; \text{for all }\;x,y \in S.$$
 where $h=\frac{f}{\alpha}$. So, from \cite{elqorachi} $\mu(x)f(\tau(x))=f(x)$ for all $x\in
 S$. Consequently, $\mu(x)f(\tau(x))=f(x)=-f(x)$ for all $x\in S$, which implies that $f=0$. This
 contradicts
  the assumption that $f\neq0$ and  this proves (5.2).\\
Equation (5.3): By replacing $y$ by $yz_0$ in (\ref{eq11}) we get
\begin{equation}\label{elq4}
\mu(yz_0)f(\tau(y)xz_{0}\tau(z_0))-f(xyz_{0}z_0)=2f(x)f(yz_0).\end{equation}
Replacing $x$ by $xz_{0}$ in (\ref{eq11})  we get
\begin{equation}\label{elq4}
\mu(y)f(\tau(y)xz_{0}z_0)-f(xyz_{0}z_0)=2f(y)f(xz_{0}).\end{equation}
Subtracting these equations results in
\begin{equation}\label{elq5}
\mu(yz_0)f(\tau(y)xz_{0}\tau(z_0))-\mu(y)f(\tau(y)xz_{0}z_0)
\end{equation}$$=2f(x)f(yz_0)-2f(y)f(xz_{0}).$$
On the other hand from (\ref{eq11}) we have
$$\mu(yz_0)f(\tau(y)xz_{0}\tau(z_0))-\mu(y)f(\tau(y)xz_{0}z_0)$$
$$=\mu(y)[\mu(z_0)f(\tau(z_0)\tau(y)xz_{0})-f(\tau(y)xz_0z_{0})]$$
$$=2\mu(y)f(z_0)f(\tau(y)x).$$
This implies that
\begin{equation}\label{elq6}
f(x)f(yz_0)-f(y)f(xz_{0})=\mu(y)f(\tau(y)x)f(z_0)\end{equation} for
all $x,y\in S.$ Since
$f(x)f(yz_0)-f(y)f(xz_{0})=-[f(y)f(xz_0)-f(x)f(yz_{0})]$, then we
deduce $\mu(y)f(\tau(y)x)f(z_0)=-\mu(x)f(\tau(x)y)f(z_0)$.
Now, by using (\ref{eqo8}) we deduce (\ref{eqoo8}).\\
Equation (5.7): By replacing $x$ by $x\tau(z_{0})$ in (\ref{eq11})
we get
\begin{equation}\label{elq7}
 \mu(y) f(\tau(y)x\tau(z_{0})z_0)-f(xy\tau(z_{0})z_0)\end{equation}$$=2f(y)f(x\tau(z_{0})).
$$ From (\ref{eqoo8}) we have
$\mu(\tau(x))f(xy\tau(z_{0})z_0)=\mu(\tau(x))f(\tau(\tau(x))(y\tau(z_{0})z_0))=-\mu(y)f(\tau(y)\tau(x)\tau(z_0)z_{0})$
and then  equation (\ref{elq7}) can be written as follows
\begin{equation}\label{elq8}
  f(\tau(y)x\tau(z_{0})z_0)+\mu(x)f(\tau(y)\tau(x)\tau(z_{0})z_0)\end{equation}
  $$=2f(y)\mu(\tau(y))f(x\tau(z_{0})).
$$ By replacing $x$ by $\tau(x)$ in (\ref{elq8}) and multiplying the result obtained by $\mu(x)$  and using  $f\neq0$ we get
 (\ref{elq10}).
\\From equations (\ref{eqoo8}) and (\ref{elq10})
we have
 $$\mu(\tau(x))f(xz_{0})=-\mu(z_0)f(\tau(z_{0})\tau(x))$$
 $$=-\mu(z_0)\mu(\tau(x))f(x\tau(z_{0})) =f(\tau(x)z_{0}).$$
 This proves (\ref{eqo11}).
 \\Equation (5.1): Replacing $x$ by $\tau(x)$ in (\ref{eq11}) we
 get
\begin{equation}\label{omar10}
 \mu(y)f(\tau(y)\tau(x)z_0)-f(\tau(x)yz_0)
 =2f(\tau(x))f(y),\;x,y\in S.\end{equation}
  We will discuss the following two possibilities.\\(1)  $\tau$ is an involutive automorphism.
  From (\ref{eqo11}) we have $$f(\tau(y)\tau(x)z_0)= f(\tau(yx)z_0)=\mu(\tau(yx))f(yxz_0)$$
   and in view of (\ref{omar10}) we obtain
 $$\mu(\tau(x))f(yxz_0)-f(\tau(x)yz_0)=2f(\tau(x))f(y),\,x,y\in S.$$
 Since
 $$\mu(\tau(x))f(yxz_0)-f(\tau(x)yz_0)=-\mu(\tau(x))[\mu(x)f(\tau(x)yz_0)-f(yxz_0)]=-\mu(\tau(x))2f(y)f(x),$$
  then we deduce that $$-2\mu(\tau(x))f(x)f(y)=2f(\tau(x))f(y)$$ for all
 $x,y\in S$. Since $f\neq0$ then we have (\ref{eqo7}).\\ (2)   $\tau$ is an involutive anti-automorphism.  Using (\ref{eqo11}) we have
 $$f(\tau(y)\tau(x)z_0)=
  f(\tau(xy)z_0)=\mu(\tau(yx))f(xyz_0)$$ and
  $f(\tau(x)yz_0)=\mu(\tau(x)y)f(\tau(y)xz_0)$.
  Now, equation (\ref{omar10}) can be written as follows
  $$\mu(\tau(x))f(xyz_0)-\mu(\tau(x)y)f(\tau(y)xz_0)=2f(\tau(x))f(y)=-\mu(\tau(x))[\mu(y)f(\tau(y)xz_0)-f(xyz_0)]$$$$
  =-\mu(\tau(x))2f(x)f(y).$$ Since $f\neq0$ then we obtain again
  (\ref{eqo7}).
 \\Equation (5.4): Putting $x=\tau(z_{0})$ in (\ref{eq11}), using
(\ref{eqo7}) we get
  $$\mu(y)f(\tau(y)\tau(z_0)z_0)-f(\tau(z_{0})yz_0) =2f(y)f(\tau(z_{0}))$$$$=-2f(y)\mu(\tau(z_0))f(z_{0}).$$
  Since
 $$\mu(y)f(\tau(y)\tau(z_{0})z_0)=-\mu(\tau(z_0)z_0)f(yz_{0}\tau(z_0))=-f(yz_{0}\tau(z_0))$$ then we
 obtain
 $$f(\tau(z_{0})yz_0)=\mu(\tau(z_0))f(y)f(z_{0})$$  for all $y\in S$. We see that we deal with (5.4).\\Equation (5.5):
  Replacing $y$ by $z_{0}$ in (\ref{eq11}) and using (\ref{eqo9}) we
 get $$\mu(z_0)f(\tau(z_{0})xz_0)-f(xz_{0}z_0)$$$$
 =2f(x)f(z_{0})=f(x)f(z_{0})-f(xz_{0}z_0).$$ Which proves (5.5).
 Equation (5.8):
  By replacing $x$ by $z_{0}$ in (\ref{eqo11})  we get
  $ \mu(z_{0})f(\tau(z_{0})z_0)=
  f(z_{0}^{2}).$ From (5.1) we have
  $f(\tau(z_{0})z_0)=-f(\tau(z_0)z_{0})$ , then we conclude that  $$ f(\tau(z_{0})z_0)=
  f(z_{0}^{2})=0,$$
 which  proves  (\ref{omar100}). This completes the proof.\end{proof}
\begin{lem} Let $f$: $S\longrightarrow \mathbb{C}$ be a non-zero solution of equation (\ref{eq11}).
 Then\\
 (1) The function defined by $$g(x)\;:=\frac{f(xz_0)}{f(z_0)}\;
\text{for} \;x\in S$$ is a non-zero  solution of the variant of
d'Alembert's functional equation (1.15). \\ (2) The function $g$
from (1) has the form $g=\frac{\chi+\mu\chi\circ \tau}{2}$, where
$\chi$ : $S\longrightarrow \mathbb{C}$, $\chi\neq 0$, is a
multiplicative function.\end{lem}
\begin{proof}
 (1)  From (\ref{eqo9}), (\ref{eqo10}), (\ref{eq11}) and the definition of $g$ we have $$(f(z_0))^{2}[g(xy)+\mu(y)g(\tau(y)x)]=
 f(z_0) \mu(y)f(\tau(y)xz_0)+f(z_0) f(xyz_{0})$$
 $$=\mu(y)\mu(z_0)f(\tau(y)xz_{0}\tau(z_0)z_0)-f(xyz_{0}z_0^{2})$$$$=
  \mu(yz_0)f(\tau(yz_0)(xz_{0})z_0)-f((xz_{0})(yz_0)z_0)$$
  $$=2f(xz_{0})f(yz_0).$$
 Dividing by $(f(z_0))^{2}$ we get $g$ satisfies the variant of d'Alembert's functional equation (1.15).
 In view of (\ref{eqo10}) and the definition of $g$ we get  $$g(z_0^{2})=\frac{f(z_0z_0^{2})}{f(z_{0})}$$
 $$=\frac{-f(z_0)f(z_{0})}{f(z_{0})}=-f(z_{0})\neq 0.$$ Then  $g$ is non-zero solution
of equation (1.15).  \\(2)  By replacing $x$ by $xz_{0}$ in
(\ref{eq11})  we get
\begin{equation}\label{haj1}
\mu(y)f(\tau(y)xz_{0}^{2})-f(xyz_{0}^{2})=2f(y)f(xz_{0}).\end{equation}
By using (\ref{eqo10})  equation (\ref{haj1}) can be written as
follows
\begin{equation}\label{haj2}
-\mu(y)f(\tau(y)x)+f(xy)=2f(y)g(x),\;x,y\in S,\end{equation} where
$g$ is the function defined above. If we replace $y$ by $yz_{0}$ in
(\ref{eq11}) we get
\begin{equation}\label{haj3}
\mu(yz_0)f(\tau(y)x\tau(z_{0})z_0)-f(xyz_{0}z_0)=2f(x)f(yz_{0}).\end{equation}
By using (\ref{eqo9}), (\ref{eqo10}) we obtain
\begin{equation}\label{haj4}
\mu(y)f(\tau(y)x)+f(xy)=2f(x)g(y),\;x,y\in S.\end{equation} By
adding (\ref{haj4}) and  (\ref{haj2}) we get that the pair $f,g$
satisfies the sine addition law $$f(xy)=f(x)g(y)+f(y)g(x)\;
\text{for all }\; x,y\in S.$$ Now, in view of [\cite{ebanks}, Lemma
3.4.] $g$ is abelian. Since $g$ is a non-zero solution of
d'Alembert's functional equation (1.15) then  from [\cite{belfkih},
Theorem 2.1 ] there exists a non-zero multiplicative function
$\chi$: $S\longrightarrow \mathbb{C}$ such that
$g=\frac{\chi+\mu\chi\circ\tau}{2}$. This completes the proof.
\end{proof}
 The following theorem is the main result of this section.
 \begin{thm} The non-zero solutions $f$ : $S\longrightarrow \mathbb{C}$ of
the functional equation (\ref{eq11}) are the functions of the form
\begin{equation}\label{elq300}
    f=\frac{\mu\chi\circ\tau-\chi}{2}\chi(z_0),
\end{equation} where $\chi$ : $S\longrightarrow \mathbb{C}$ is a
multiplicative function such that $\chi(z_0)\neq 0$ and
$\mu(z_0)\chi(\tau(z_0))=-\chi(z_0)$.
\\If $S$ is a
topological semigroup and that $\tau$ : $S\longrightarrow S$ and
$\mu$: $S\longrightarrow \mathbb{C}$ are continuous then the
non-zero solution $f$ of equation (\ref{eq11}) is continuous, if and
only if $\chi$ is continuous.
\end{thm}\begin{proof} Simple computations show that $f$ defined by (\ref{elq300}) is a solution of (\ref{eq11}). Conversely,
 let $f$ : $S\longrightarrow \mathbb{C}$ be a
non-zero solution of the functional equation (\ref{eq11}). By
putting $y=z_0$ in (\ref{eq11}) we get
\begin{equation}\label{elqorachi1222}
f(x)=\frac{\mu(z_0)f(\tau(z_0)xz_0)-f(xz_0z_0)}{2f(z_0)}\end{equation}
$$=\frac{1}{2}( \mu(z_0)g(\tau(z_0)x)-g(xz_0)),$$
 where g is the function defined  by  $g(x)=\frac{f(xz_0)}{f(z_0)}$ and that from Lemma 5.2  has the form  $g=\frac{\chi+\mu\chi\circ \tau}{2}$,
 where $\chi$ : $S\longrightarrow \mathbb{C}$, $\chi\neq 0$ is a
multiplicative function. Substituting this into
(\ref{elqorachi1222}) we find that $f$ has the form
\begin{equation}\label{equo1àà}
    f=\frac{\chi(z_0)-\mu(z_0)\chi(\tau(z_0))}{2}\frac{ \mu\chi\circ\tau-\chi}{2}.
\end{equation}Furthermore, from (\ref{eqo11})  $f$ satisfies
$\mu(x) f(\tau(x)z_0)=f(xz_0)$ for all $x\in S$. By applying the
last expression of $f$ in  (\ref{eqo11}) we get after computations
that
$$[\mu(z_0)\chi(\tau(z_0))+\chi(z_0)][\chi-\mu\chi\circ\tau]=0.$$ Since $\chi\neq  \mu \chi\circ\tau$, we
obtain
 $\mu(z_0)\chi(\tau(z_0))+\chi(z_0)=0$ and  then from (\ref{equo1àà})
 we have
$$f=\frac{ \mu\chi\circ\tau-\chi}{2}\chi(z_0).$$
For the topological statement we use [\cite{07}, Theorem 3.18(d)].
This completes the proof. \end{proof}


\begin{thebibliography}{20}
\bibitem{preprint}  Bouikhalene, B. and  Elqorachi, E. An extension of Van Vleck's functional equation for the sine. Acta Math. Hungarica,
\textbf{150} (2016), Issue 1, 258-267.
\bibitem{belfkih} Belfakih, K., and  Elqorachi, E., $\chi$-Wilson's functional equations for
vector and $2\times 2$-matrix valued functions, preprint, (2016)
\bibitem{davison}  Davison, T.M.K., D'Alembert's functional equation
on topological monoids. Publ. Math. debrecen \textbf{75} (2009),
Issue 1/2, 41-66.
\bibitem{elqorachi}  Elqorachi, E. and  Redouani, A., Solutions and stability of a variant
of Wilson's functional equation.  arXiv:1505.06512v1 [math.CA]
(2015), Demonstratio Math. (to appear).
\bibitem{ebanks}  Ebanks, B. R. and  Stetk{\ae}r, H., D'Alembert's other functional equation on monoids with involution.
Aequationes Math. \textbf{89} (2015), Issue 1, 187-206.
 \bibitem{K}  Kannappan, Pl., A functional equation for the cosine. Canad. Math.
Bull. \textbf{2} (1968), 495-498.
\bibitem{P} A.M. Perkins, and  Sahoo, P.K., On two functional equations with
involution on groups related to sine and cosine functions.
Aequationes Math. \textbf{89} (2015), Issue 5,   1251-1263.
\bibitem{stkan}  Stetk{\ae}r, H.,
Kannappan's functional equation on semigroups with involution.
Semigroup Forum, DOI: 10.1007/s00233-015-9756-7, (2016), 1-14.
\bibitem{St3}  Stetk\ae r, H., Van Vleck's functional equation for the
sine. Aequationes Math.  \textbf{90} (2016), Issue 1,  25-34.
\bibitem{St100}  Stetk\ae r, H., A variant of
d'Alembert's functional equation. Aequationes Math.  \textbf{89}
(2015), Issue 3, 657-662.
\bibitem{St1}  Stetk\ae r, H.,  d'Alembert's functional equation on groups, Recent
developments in functional equations and inequalities. pp. 173-191,
Banach Center Publ., vol. \textbf{99}. Polish Acad. Sci. Inst. Math.
 Warsaw. (2013)
\bibitem{07}  Stetk{\ae}r, H.:
Functional Equations on Groups, World Scientific Publishing Co,
Singapore 2013.
\bibitem{stetkaer1} Stetk{\ae}r, H., On a variant of Wilson's functional
equation on groups, Aequationes Math. \textbf{68} (2004), no. 3,
160-176.
\bibitem{V1}  Vleck Van, E.B., A
functional equation for the sine. Ann. Math. Second Ser. \textbf{11}
(4), 161-165 (1910).
 \bibitem{V2}  Vleck  Van, E.B., A functional equation for the
sine, Additional note. Ann. Math. Second Ser. \textbf{13} (1/4), 154
(1911-1912).
\end{thebibliography}
\end{document}